\documentclass[11pt]{amsart}
\usepackage{amsmath,amssymb,amsfonts,mathtools}
\usepackage{float} 
\usepackage{amscd}
\usepackage[margin=3cm]{geometry}
\usepackage{tikz}
\usepackage{graphicx}

\numberwithin{equation}{section}

\newtheorem{theorem}{Theorem}[section]
\newtheorem{proposition}[theorem]{Proposition}
\newtheorem{corollary}[theorem]{Corollary}
\newtheorem{lemma}[theorem]{Lemma}
\theoremstyle{definition}
\newtheorem{definition}[theorem]{Definition}
\newtheorem{remark}[theorem]{Remark}
\newtheorem{example}[theorem]{Example}

\begin{document}

\title[Breaker Chains and Splitting]{ Breaker Chains and Splitting of Interval Submodules in Zigzag Modules}

\author[K. Gongopadhyay and R. Singh]{Krishnendu Gongopadhyay and Ramveer Singh}

\address{Indian Institute of Science Education and Research (IISER) Mohali,
Knowledge City, Sector 81, S.A.S. Nagar 140306, Punjab, India}
\email{krishnendu@iisermohali.ac.in}

\address{Indian Institute of Science Education and Research (IISER) Mohali,
Knowledge City, Sector 81, S.A.S. Nagar 140306, Punjab, India}
\email{ramveersingh5456@gmail.com}

\subjclass[2020]{Primary 16G20; Secondary 55N31}
\keywords{Zigzag modules, interval decomposition, breaker pairs, breaker chains, interval submodules, quiver representations, zigzag persistence}

\begin{abstract}
Finite-dimensional zigzag modules decompose into interval modules.  For an
ordinary persistence module, every fully supported interval submodule of a
restriction splits as a direct summand, whereas this need not hold for a
zigzag module.  We study this failure of splitting.  A breaker pair is a
local configuration of two overlapping interval summands with a backward
arrow at the left boundary of the second interval and a forward arrow at the
right boundary of the first.  A single breaker pair gives a local
nonsplitting interval submodule, but it need not control a prescribed larger
restriction.  We therefore introduce breaker chains.  Our main result shows
that a restriction $V[p,q]$ contains a nonsplitting submodule isomorphic to
$I[p,q]$ if and only if its interval decomposition contains a breaker chain
spanning $[p,q]$.  Equivalently, the proper interval summands $J$ for which
$\operatorname{Hom}(I[p,q],J)\ne0$ have supports covering $[p,q]$.  This
cover formulation gives a constructive description of nonsplitting
embeddings, a trichotomy for their splitting behavior, and an
$O(m\log m)$ algorithm from an oriented barcode with $m$ intervals.
\end{abstract}

\maketitle

\section{Introduction}

Persistence modules and zigzag modules arise naturally in topological data
analysis and in the representation theory of quivers
\cite{AuslanderReitenSmalo,CarlssonSurvey,EdelsbrunnerHarer,Oudot}.  A
persistence module over a field $K$ consists of a sequence of vector spaces
and linear maps all oriented in one direction.  Zigzag modules allow both
forward and backward arrows and are useful when the underlying data does not
evolve monotonically \cite{CarlssonDeSilvaZigzag}.

A fundamental structural result, ultimately a consequence of Gabriel's
theorem for quivers of type $A_n$, states that every finite-dimensional
zigzag module decomposes as a direct sum of interval modules
\cite{Gabriel,CarlssonDeSilvaZigzag}.  Modern treatments and extensions of
interval decomposition can be found in
\cite{Botnan2017,BotnanCrawleyBoevey2020,CrawleyBoevey}.  The interval
summands are the indecomposable building blocks underlying the barcode.

There is a basic distinction between ordinary persistence and zigzag
persistence.  If $V[p,q]$ is a restriction of an ordinary persistence module,
then every submodule isomorphic to the full interval $I[p,q]$ splits in
$V[p,q]$.  In a zigzag module this can fail.  The simplest example is
\[
K \longleftarrow K^2 \longrightarrow K,
\]
where the two arrows are the projections onto the first and second
coordinates. Its interval decomposition is
\[
I[1,2]\oplus I[2,3],
\]
so the barcode contains no interval $I[1,3]$. Nevertheless, the diagonal
line in the middle vertex determines a submodule isomorphic to $I[1,3]$,
giving a representative that persists across all three vertices even though
no single barcode interval accounts for it. This copy of $I[1,3]$ does not
split.
The local mechanism behind this example is a pair of overlapping interval
summands whose boundary orientations point in opposite directions: the
second interval is entered by a backward arrow, while the first interval is
left by a forward arrow.  We call such a configuration a \emph{breaker
pair}.  A breaker pair always produces a nonsplitting interval over the hull
of the two supports.  However, one breaker pair need not produce a fully
supported submodule on a larger prescribed restriction.  For example, for
the zigzag type
\[
1\longleftarrow 2\longrightarrow 3\longrightarrow 4
\]
the module
\[
I[1,2]\oplus I[2,3]\oplus I[4,4]
\]
contains the breaker pair $I[1,2],I[2,3]$, but it contains no submodule
isomorphic to $I[1,4]$.  Thus a fixed interval $[p,q]$ requires a global
configuration built from successive local obstructions.

This leads to the notion of a \emph{breaker chain}: a sequence of interval
summands in which every consecutive pair is a breaker pair.  A chain spans
$[p,q]$ when its first interval begins at $p$ and its last interval ends at
$q$.  Our main result is the following.

\begin{theorem}\label{thm:main}
Let $V[p,q]$ be a restriction of a finite-dimensional zigzag module, and fix
an interval decomposition of $V[p,q]$.  The following are equivalent.
\begin{enumerate}
\item There exists a monomorphism
\[
\iota:I[p,q]\longrightarrow V[p,q]
\]
whose image is not a direct summand of $V[p,q]$.
\item The interval decomposition of $V[p,q]$ contains a breaker chain
spanning $[p,q]$.
\end{enumerate}
\end{theorem}

By the uniqueness of interval decomposition, condition~(2) is independent of
the chosen realization of the direct sum; see Remark~\ref{rem:decomp-independence}.

We begin with the classical morphism calculus for interval representations of
type-$A$ quivers, applied here to the spaces
\[
\operatorname{Hom}(I[p,q],I[b,d])
\qquad\text{and}\qquad
\operatorname{Hom}(I[b,d],I[p,q]).
\]
In the present setting, the resulting boundary-orientation criterion leads to
a global cover characterization: a monomorphism
\[
I[p,q]\hookrightarrow V[p,q]
\]
exists exactly when the supports of the interval summands receiving a nonzero
morphism from $I[p,q]$ cover $[p,q]$, while a nonsplitting monomorphism exists
exactly when the proper such summands already cover $[p,q]$. This boundary criterion is known in the literature and can also be obtained  from the block
decomposition of the two-parameter extension of a zigzag module
\cite{BotnanLesnick}; see Remark~\ref{rem:blocks} for the comparison.  What is
specific to the present work is the global configuration that these local
conditions assemble into.

The key point is that these local morphism conditions lead to a global
obstruction to splitting. Starting with a proper admissible cover and removing
redundant intervals, the remaining intervals form a spanning breaker chain.
Conversely, every spanning breaker chain gives a proper admissible cover.
Thus breaker chains characterize the obstruction to splitting. The cover
description also gives an explicit construction of a nonsplitting copy and an
efficient interval-sweep algorithm. As a consequence, ordinary persistence modules have no breaker chains, since
all arrows point in the same direction. Hence every fully supported interval
submodule of a restriction splits. A single breaker pair is the simplest case
of the theorem and gives the usual diagonal construction on the hull of the two
intervals.

The paper is organised as follows. Section 2 reviews basic notions. Section 3 develops breaker pairs and chains, proves the nonsplitting criterion, and derives an $O(m\log m)$ barcode algorithm.

\section{Persistence and zigzag modules}\label{2}

Let $K$ be a field and let $n\in\mathbb N$.

\begin{definition}[Persistence module]
A \emph{persistence module} of length $n$ over $K$ is a sequence
\[
V_1\xrightarrow{f_1}V_2\xrightarrow{f_2}\cdots
\xrightarrow{f_{n-1}}V_n
\]
of finite-dimensional vector spaces and linear maps.
\end{definition}

\begin{definition}[Zigzag module]
A \emph{zigzag module} $V$ of length $n$ consists of finite-dimensional
vector spaces $V_1,\dots,V_n$ such that for each $1\le i<n$ exactly one of
\[
V_i\longrightarrow V_{i+1},
\qquad
V_i\longleftarrow V_{i+1}
\]
is specified.  We always compare modules having the same arrow type.
\end{definition}

\begin{definition}[Interval module]
For $1\le b\le d\le n$, the \emph{interval module} $I[b,d]$ is defined by
\[
I[b,d]_i=
\begin{cases}
K,& b\le i\le d,\\
0,& \text{otherwise},
\end{cases}
\]
with identity maps between consecutive nonzero spaces and zero maps whenever
one endpoint is zero.  Its support is $[b,d]$.
\end{definition}

\begin{definition}[Submodule and splitting]
A submodule $W\subseteq V$ is a collection of subspaces $W_i\subseteq V_i$
preserved by all structure maps.  A submodule $W$ is a \emph{direct
summand} if there is a submodule $W'$ such that $V=W\oplus W'$.  Equivalently,
the inclusion $W\hookrightarrow V$ admits a retraction $V\to W$.
\end{definition}

\begin{theorem}[Gabriel]\label{thm:gabriel}
For a fixed zigzag type of length $n$, the indecomposable finite-dimensional
modules are precisely the interval modules $I[b,d]$, $1\le b\le d\le n$.
Consequently every finite-dimensional zigzag module is a finite direct sum of
interval modules.  The multiset of interval summands is unique up to
isomorphism and permutation.
\end{theorem}

\begin{definition}[Restriction]
For $p\le q$, the \emph{restriction} $V[p,q]$ is obtained from $V$ by keeping
only the vector spaces and arrows with indices between $p$ and $q$.
\end{definition}

\begin{definition}[Fully supported interval submodule]
A submodule $W\subseteq V[p,q]$ is called a \emph{fully supported interval
submodule} if $W\cong I[p,q]$.  Thus $W_i$ is one-dimensional for every
$p\le i\le q$, and every structure map of $W$ is an isomorphism.
\end{definition}

For comparison with the zigzag situation, we record the standard
equioriented criterion, see \cite{CarlssonDeSilvaZigzag}. 

\begin{proposition}\label{prop:persistence}
Let
\[
V_1\xrightarrow{f_1}V_2\xrightarrow{f_2}\cdots
\xrightarrow{f_{n-1}}V_n
\]
be a persistence module and let $1\le p\le q\le n$.  The following are
equivalent.
\begin{enumerate}
\item The composite $f_{q-1}\cdots f_p:V_p\to V_q$ is nonzero.
\item There exist nonzero $x_i\in V_i$, $p\le i\le q$, with
$x_{i+1}=f_i(x_i)$.
\item The restriction $V[p,q]$ contains a submodule isomorphic to $I[p,q]$.
\item The restriction $V[p,q]$ contains a direct summand isomorphic to
$I[p,q]$.
\end{enumerate}
\end{proposition}

Proposition~\ref{prop:persistence} is an existence statement: the existence
of a full interval submodule implies the existence of a full interval
summand.  The question studied below is stronger and concerns whether a
\emph{given embedded copy} $I[p,q]\hookrightarrow V[p,q]$ splits.  In the
equioriented case every such embedding does split; this will follow from Theorem~\ref{thm:main}, see 
Corollary~\ref{cor:persistence-splitting}.

\section{Breaker pairs and breaker chains}\label{3}

We now fix a zigzag type and work inside a restriction $V[p,q]$.

\begin{definition}[Breaker pair]\label{def:breaker-pair}
Let
\[
J_1\cong I[b_1,d_1],
\qquad
J_2\cong I[b_2,d_2]
\]
be interval summands.  The ordered pair $(J_1,J_2)$ is a \emph{breaker pair}
if
\[
b_1<b_2\le d_1<d_2,
\]
the arrow at the left boundary of $J_2$ is backward,
\[
V_{b_2-1}\longleftarrow V_{b_2},
\]
and the arrow at the right boundary of $J_1$ is forward,
\[
V_{d_1}\longrightarrow V_{d_1+1}.
\]
\end{definition}

\begin{figure}[h]
\centering
\begin{tikzpicture}[scale=0.9]
\draw[thick] (0,0)--(10,0);
\draw[very thick] (1,1)--(6,1);
\draw[very thick] (3,2)--(8,2);
\node[left] at (1,1) {$J_1$};
\node[left] at (3,2) {$J_2$};
\node[below] at (1,0) {$b_1$};
\node[below] at (3,0) {$b_2$};
\node[below] at (6,0) {$d_1$};
\node[below] at (8,0) {$d_2$};
\draw[dashed] (3,-0.2)--(3,2.25);
\draw[dashed] (6,-0.2)--(6,2.25);
\draw[->,thick] (3,-0.55)--(2.2,-0.55);
\draw[->,thick] (6,-0.9)--(6.8,-0.9);
\node[below] at (2.6,-0.55) {backward};
\node[below] at (6.4,-0.9) {forward};
\end{tikzpicture}
\caption{A breaker pair.  The two supports overlap, the second interval is
entered by a backward arrow, and the first interval is left by a forward
arrow.}
\label{fig:breaker-pair}
\end{figure}
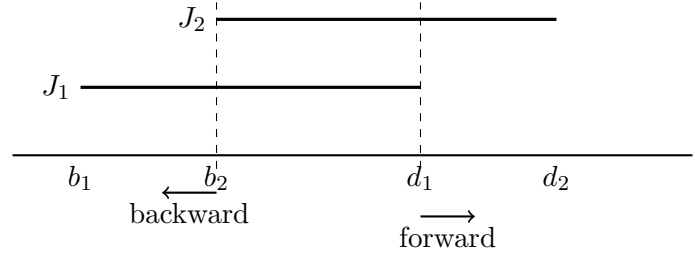

\begin{definition}[Breaker chain]\label{def:breaker-chain}
A sequence of interval summands
\[
J_1,\dots,J_r,
\qquad r\ge2,
\]
is a \emph{breaker chain} if $(J_j,J_{j+1})$ is a breaker pair for every
$1\le j<r$.  Writing $J_j\cong I[b_j,d_j]$, the chain \emph{spans $[p,q]$}
if
\[
b_1=p,
\qquad
d_r=q.
\]
\end{definition}

Thus a spanning breaker chain satisfies
\[
b_1=p,\qquad d_r=q,
\]
and, for every $1\le j<r$,
\[
b_j<b_{j+1}\le d_j<d_{j+1}.
\]
It has a backward arrow at every $b_j$ for $j\ge2$ and a forward arrow at
every $d_j$ for $j<r$.

\medskip 
The following lemma is a specialization of the standard description of
morphisms between spread modules, see \cite[Proposition 5.5]{BBH2}, \cite[Prop. 2.7]{BBH} (numbering as in the  arXiv versions). 
We note the resulting boundary-orientation criterion for zigzag interval
modules and include a proof for completeness.

\begin{lemma}[Boundary criterion for interval morphisms]\label{lem:hom}
Let $U=I[p,q]$ and $J=I[b,d]$, where
\(
p\le b\le d\le q.
\) Here $U$ and $J$ are regarded as modules over the restricted zigzag type on
$[p,q]$; since both vanish outside $[p,q]$, the Hom-spaces are unchanged if
they are instead viewed as modules over the full zigzag type of length $n$.

Then the following holds. 
\begin{enumerate}
\item $\operatorname{Hom}(U,J)\ne0$ if and only if
\[
b>p\Longrightarrow V_{b-1}\longleftarrow V_b,
\qquad
d<q\Longrightarrow V_d\longrightarrow V_{d+1}.
\]
When these conditions hold,
$\operatorname{Hom}(U,J)\cong K$; otherwise it is zero.
\item $\operatorname{Hom}(J,U)\ne0$ if and only if
\[
b>p\Longrightarrow V_{b-1}\longrightarrow V_b,
\qquad
d<q\Longrightarrow V_d\longleftarrow V_{d+1}.
\]
When these conditions hold,
$\operatorname{Hom}(J,U)\cong K$; otherwise it is zero.
\end{enumerate}
\end{lemma}
\begin{proof}
For each $i$, let $\alpha_i$ denote the arrow joining $i$ and $i+1$.
For a morphism $\varphi=(\varphi_i)\colon X\to Y$, naturality at an arrow
$\alpha\colon s\to t$ means
\[
\varphi_t\circ X(\alpha)=Y(\alpha)\circ\varphi_s. \tag{$*$}
\]

Since $J_i=0$ for $i\notin[b,d]$, all components of a morphism involving $J$
are zero outside $[b,d]$. For $i\in[b,d]$, write
$\varphi_i=\lambda_i\operatorname{id}_K$.

If $b\le i<d$, then both endpoints of $\alpha_i$ lie in $[b,d]$. Hence both
$U$ and $J$ assign $\operatorname{id}_K$ to $\alpha_i$, and naturality gives
$\lambda_i=\lambda_{i+1}$, regardless of the direction of the arrow.
Since $[b,d]$ is connected, all the $\lambda_i$ are equal. Write their common
value as $\lambda$. Thus, in either direction, a morphism is determined by one
scalar, so each Hom-space has dimension at most $1$.

It remains to check the boundary arrows. All other arrows give no condition:
if both endpoints lie outside $[b,d]$, or if one endpoint lies outside
$[p,q]$, then equation $(*)$ is simply $0=0$.

Thus the only possible conditions come from the arrow between $b-1$ and $b$
when $b>p$, and the arrow between $d$ and $d+1$ when $d<q$.

Consider first a morphism $\varphi\colon U\to J$. At a boundary edge, let
$j^*\in[b,d]$ be the endpoint inside $[b,d]$ and
$j\in[p,q]\setminus[b,d]$ the other endpoint. Then
$U_j=K$, $J_j=0$, $\varphi_j=0$, and
$\varphi_{j^*}=\lambda\operatorname{id}_K$.
If the boundary arrow points from $j^*$ to $j$, equation $(*)$ is automatic.
If it points from $j$ to $j^*$, then $(*)$ forces $\lambda=0$.

Hence a nonzero morphism $U\to J$ exists exactly when the boundary arrows
point out of $[b,d]$, that is,
\[
b>p\Longrightarrow V_{b-1}\longleftarrow V_b,
\qquad
d<q\Longrightarrow V_d\longrightarrow V_{d+1}.
\]
In this case the morphism is determined by the arbitrary scalar $\lambda$, so
$\operatorname{Hom}(U,J)\cong K$.

For a morphism $J\to U$, the same argument gives the opposite boundary
conditions:
\[
b>p\Longrightarrow V_{b-1}\longrightarrow V_b,
\qquad
d<q\Longrightarrow V_d\longleftarrow V_{d+1}.
\]
Again the morphism is determined by one scalar, and therefore
$\operatorname{Hom}(J,U)\cong K$ whenever these conditions hold.
\end{proof}

\begin{definition}[Internal boundary]\label{def:internal}
Fix a restriction $[p,q]$ and an interval $[b,d] \subseteq [p,q]$. We call an
endpoint of $[b,d]$ \emph{internal} if it lies in the interior of $[p,q]$, i.e.\
$b > p$ or $d < q$ respectively; the corresponding arrow, between $b-1$ and $b$
or between $d$ and $d+1$, is the \emph{internal boundary arrow}.\end{definition}

\begin{definition}[Proper interval summand]\label{def:proper}
Relative to the fixed restriction $[p,q]$, an interval summand
$J\cong I[b,d]$ is called \emph{proper} if
\[
[b,d]\subsetneq[p,q],
\]
equivalently, if $J\not\cong I[p,q]$.
\end{definition}

\begin{definition}[Admissible interval summand]\label{def:admissible}
Let  $U=I[p,q]$.  An interval summand $J$ of $V[p,q]$ is called
\emph{admissible} if
\(
\operatorname{Hom}(U,J)\ne0.
\)
By Lemma~\ref{lem:hom}, admissibility is determined entirely by the support
of $J$ and the orientations of its internal boundary arrows.
\end{definition}
\begin{remark}[Relation to the two-parameter block picture]\label{rem:blocks}
Lemma~\ref{lem:hom} also has a natural interpretation in the two-parameter
framework of Botnan and Lesnick \cite{BotnanLesnick}. They extend zigzag
modules to $\mathbb{R}^{\mathrm{op}}\times\mathbb{R}$ and show that they
decompose into \emph{blocks}: quadrants or horizontal/vertical bands, with
the type determined by the internal boundary arrows as in
Definition~\ref{def:internal}. The general Hom criterion for interval
modules then specializes to Lemma~\ref{lem:hom}. Thus, the lemma is not new, as already noted for spread modules in \cite{BBH2,BBH};
we give a direct proof because it is elementary and self-contained.

The main content of Theorem~\ref{thm:main} is instead combinatorial: passing
from a proper admissible cover to an inclusion-minimal subcover and showing
that it is a spanning breaker chain. This argument, as well as
Corollary~\ref{cor:trichotomy} and Proposition~\ref{prop:algorithm}, is
independent of the chosen framework. The block viewpoint nevertheless
suggests extensions to infinite zigzags and $\mathbb{R}$-indexed level-set
persistence, which we leave open.
\end{remark}

\begin{example}\label{ex:single-pair-not-global}
A single breaker pair is only a local obstruction.  Consider the zigzag type
\[
1\longleftarrow2\longrightarrow3\longrightarrow4
\]
and
\[
V=I[1,2]\oplus I[2,3]\oplus I[4,4].
\]
The pair $(I[1,2],I[2,3])$ is a breaker pair, but $V$ has no submodule
isomorphic to $I[1,4]$: the structure map $V_3\to V_4$ is zero.  Hence a
single breaker pair cannot characterize nonsplitting on an arbitrary larger
restriction.
\end{example}

Fix an interval decomposition
\[
V[p,q]=\bigoplus_{\alpha\in A}J_\alpha
\]
and let $\pi_\alpha:V[p,q]\to J_\alpha$ denote the canonical projection.
For a monomorphism
\[
\iota:U=I[p,q]\longrightarrow V[p,q],
\]
let 
\[
\iota_\alpha=\pi_\alpha\circ\iota.
\]
We call $J_\alpha$ \emph{active} for $\iota$ if $\iota_\alpha\ne0$.

\begin{lemma}[Active intervals cover]\label{lem:cover}
The supports of the active interval summands cover $[p,q]$.
\end{lemma}

\begin{proof}
For each $i\in[p,q]$, the component
\[
\iota_i:K\longrightarrow V_i
\]
is injective and hence nonzero.  Under the identification
\[
V_i=\bigoplus_{\alpha\in A}(J_\alpha)_i,
\]
the map $\iota_i$ is the tuple
\[
\iota_i=\bigl((\iota_\alpha)_i\bigr)_{\alpha\in A}.
\]
Since this tuple is nonzero, at least one component
$(\iota_\alpha)_i$ is nonzero.  Then $J_\alpha$ is
active and its support contains $i$.
\end{proof}

\begin{lemma}[Splitting criterion]\label{lem:splitting}
The monomorphism $\iota:U\to V[p,q]$ splits if and only if there is an active
summand $J_\alpha\cong U$.
\end{lemma}

\begin{proof}
If $J_\alpha\cong U$ is active, then
\[
\pi_\alpha\iota:U\longrightarrow J_\alpha\cong U
\]
is a nonzero endomorphism of an interval module.  Hence it is multiplication
by a nonzero scalar and therefore an isomorphism.  Composing $\pi_\alpha$
with its inverse gives a retraction of $\iota$.

Conversely, suppose $\iota$ splits, and let
\[
r:V[p,q]\longrightarrow U
\]
be a retraction.  Write $r_\alpha=r|_{J_\alpha}$.  Then
\[
\operatorname{id}_U=r\iota=\sum_{\alpha\in A}r_\alpha\iota_\alpha.
\]
If $J_\alpha$ is inactive, the corresponding term is zero.  Suppose
$J_\alpha=I[b,d]$ is active and proper.  Activity means
$\iota_\alpha\ne0$, hence
\[
\operatorname{Hom}(U,J_\alpha)\ne0.
\]
By Lemma~\ref{lem:hom}(1), every internal left boundary of $J_\alpha$ is
therefore backward and every internal right boundary is forward.  Since
$J_\alpha$ is proper, at least one boundary is internal.  At any such
boundary, the orientation required by Lemma~\ref{lem:hom}(2) for a nonzero
map $J_\alpha\to U$ is the opposite orientation.  Consequently
\[
\operatorname{Hom}(J_\alpha,U)=0,
\]
and hence $r_\alpha=0$.  Therefore the identity above cannot hold unless some
active summand is isomorphic to $U$.
\end{proof}

\begin{proposition}[Cover characterization]\label{prop:cover-characterization}
Let
\[
V[p,q]=\bigoplus_{\alpha\in A}J_\alpha,
\qquad U=I[p,q].
\]
Then:
\begin{enumerate}
\item There exists a monomorphism $U\to V[p,q]$ if and only if the supports
of the admissible interval summands cover $[p,q]$.
\item There exists a nonsplitting monomorphism $U\to V[p,q]$ if and only if
the supports of the proper admissible interval summands cover $[p,q]$.
\end{enumerate}
\end{proposition}

\begin{proof}
Suppose first that $\iota:U\to V[p,q]$ is a monomorphism.  Every active
summand is admissible by definition, and Lemma~\ref{lem:cover} shows that the
active supports cover $[p,q]$.  This proves the necessity in part~(1).  If
$\iota$ does not split, Lemma~\ref{lem:splitting} shows that no active
summand is isomorphic to $U$; hence all active summands are proper, proving
the necessity in part~(2).

Conversely, let $\mathcal C$ be any collection of admissible interval
summands whose supports cover $[p,q]$.  For each $J\in\mathcal C$, choose a
nonzero morphism
\[
\phi_J:U\longrightarrow J.
\]
Their direct sum defines
\[
\Phi=(\phi_J)_{J\in\mathcal C}:
U\longrightarrow\bigoplus_{J\in\mathcal C}J
\hookrightarrow V[p,q].
\]
At each vertex $i\in[p,q]$, some $J\in\mathcal C$ contains $i$, and a
nonzero map $U\to J$ is nonzero at every vertex of the support of $J$.
Therefore $\Phi_i\ne0$.  Since $U_i$ is one-dimensional, $\Phi_i$ is
injective for every $i$, so $\Phi$ is a monomorphism.  This proves the
sufficiency in part~(1).

If all members of $\mathcal C$ are proper, then no summand isomorphic to
$U$ is active for $\Phi$.  Lemma~\ref{lem:splitting} implies that $\Phi$
does not split, proving part~(2).
\end{proof}
\begin{corollary}\label{cor:homcount}
Let $V[p,q]=\bigoplus_{\alpha\in A}J_\alpha$ be an interval decomposition and
$U=I[p,q]$. Then
\[
  \dim_K \operatorname{Hom}\bigl(U,\,V[p,q]\bigr)
  \;=\; \#\{\alpha\in A : J_\alpha \text{ is admissible}\},
\]
counted with multiplicity. In particular this count depends only on $V[p,q]$,
not on the chosen decomposition.
\end{corollary}

\begin{proof}
Since $\operatorname{Hom}(U,-)$ commutes with finite direct sums,
$\operatorname{Hom}(U,V[p,q])\cong\bigoplus_{\alpha}\operatorname{Hom}(U,J_\alpha)$.
By Lemma~\ref{lem:hom}(1) each term is $K$ if $J_\alpha$ is admissible and $0$
otherwise. The last assertion follows because the left-hand side is defined
without reference to a decomposition.
\end{proof}
\begin{remark}\label{rem:minimal-cover-chain}
A breaker chain should not be identified with a minimal cover as an object.
Rather, Proposition~\ref{prop:cover-characterization} says that the existence
of a nonsplitting copy is equivalent to the existence of a proper admissible
cover.  Choosing an inclusion-minimal such cover produces a spanning breaker
chain in the proof below.  Conversely, every spanning breaker chain is a
proper admissible cover.  Thus the two notions are equivalent at the level of
existence, while a given breaker chain need not itself be an inclusion-minimal
cover if it is not a spanning breaker chain.
\end{remark}

We can now prove the main theorem.

\begin{proof}[\bf Proof of Theorem~\ref{thm:main}]
If $p=q$, every one-dimensional subspace of the vector space $V_p$ has a
complement, so both statements are false.  Assume $p<q$.

Suppose first that there exists a nonsplitting monomorphism
$U=I[p,q]\to V[p,q]$.  By
Proposition~\ref{prop:cover-characterization}(2), the proper admissible
summands have supports covering $[p,q]$.  Choose an inclusion-minimal
subcollection of proper admissible summands whose supports still cover
$[p,q]$, and write them as
\[
J_1=I[b_1,d_1],\dots,J_r=I[b_r,d_r]
\]
ordered by increasing left endpoint.  Necessarily $r\ge2$: if $r=1$, its
support would have to contain both $p$ and $q$, hence would equal $[p,q]$,
contradicting properness.

Minimality implies
\[
b_1<b_2<\cdots<b_r
\quad\text{and}\quad
d_1<d_2<\cdots<d_r.
\]
Indeed, if $b_i\le b_j$ but $d_i\ge d_j$ for $i\ne j$, then the support of
one chosen interval is contained in that of another, contradicting
minimality.  Since the chosen intervals cover the two endpoints,
\[
b_1=p,
\qquad
d_r=q.
\]
They also satisfy
\[
b_{j+1}\le d_j+1
\qquad(1\le j<r),
\]
otherwise there would be a gap in the cover.

We claim that equality cannot occur.  If
$b_{j+1}=d_j+1$, then $J_j$ is admissible and ends before $q$, so
Lemma~\ref{lem:hom}(1) forces
\[
V_{d_j}\longrightarrow V_{d_j+1}.
\]
On the other hand, $J_{j+1}$ is admissible and begins after $p$, so the
same lemma forces the very same arrow to be
\[
V_{d_j}\longleftarrow V_{d_j+1},
\]
a contradiction.  Hence
\[
b_{j+1}\le d_j.
\]

Finally, admissibility and Lemma~\ref{lem:hom}(1) give a backward arrow entering
$b_{j+1}$ and a forward arrow leaving $d_j$.  Thus
\[
b_j<b_{j+1}\le d_j<d_{j+1},
\]
and $(J_j,J_{j+1})$ is a breaker pair for every $j$.  Therefore
$J_1,\dots,J_r$ is a breaker chain spanning $[p,q]$.

Conversely, suppose
\[
J_1=I[b_1,d_1],\dots,J_r=I[b_r,d_r]
\]
is a breaker chain spanning $[p,q]$.  The breaker orientations and
Lemma~\ref{lem:hom}(1) show that every $J_j$ is admissible.  Since
$r\ge2$ and the endpoints strictly advance along the chain, every $J_j$ is
proper.  The chain supports cover $[p,q]$.  Therefore
Proposition~\ref{prop:cover-characterization}(2) gives a nonsplitting
monomorphism $U\to V[p,q]$.  This proves the converse implication.
\end{proof}

\begin{remark}[Independence of the interval decomposition]\label{rem:decomp-independence}
For a fixed zigzag type, Theorem~\ref{thm:gabriel} gives uniqueness of the
multiset of indecomposable interval summands up to isomorphism and
permutation.  Hence the multiset of supports $[b,d]$ is determined by the
barcode, while the orientations of all boundary arrows are determined by the
fixed zigzag type.  Consequently the statements that the decomposition
contains an admissible cover or a spanning breaker chain are independent of
the chosen realization of the direct-sum decomposition.  They depend only on
the oriented barcode of $V[p,q]$.
\end{remark}

\begin{remark}
Theorem~\ref{thm:main} separates the local and global aspects of the
obstruction.  A breaker pair is the elementary local configuration.  To
obstruct splitting of a full interval on a prescribed restriction $[p,q]$,
these local configurations must link together into a chain spanning the
whole restriction.
\end{remark}
\subsection{Consequence  of the breaker-chain criterion}
\begin{corollary}[Local breaker-pair obstruction]\label{cor:pair}
Let $(J_1,J_2)$ be a breaker pair with
\[
J_1\cong I[b_1,d_1],
\qquad
J_2\cong I[b_2,d_2].
\]
Then $J_1\oplus J_2$, restricted to $[b_1,d_2]$, contains a nonsplitting
submodule isomorphic to $I[b_1,d_2]$.  Moreover, $J_1\oplus J_2$ has no
direct summand isomorphic to $I[b_1,d_2]$.
\end{corollary}

\begin{proof}
The two intervals form a breaker chain spanning $[b_1,d_2]$, so the first
statement follows from Theorem~\ref{thm:main}.  For the second statement,
$J_1$ and $J_2$ are the indecomposable summands of $J_1\oplus J_2$, up to
permutation, and neither is isomorphic to $I[b_1,d_2]$.  The uniqueness part
of Theorem~\ref{thm:gabriel} therefore excludes such a direct summand.
\end{proof}

\begin{remark}[Explicit diagonal construction]\label{rem:diagonal}
Corollary~\ref{cor:pair} has the concrete form used in the basic example.
Choose compatible nonzero generators $x_i$ for $J_1$ and $y_i$ for $J_2$.
Then the one-dimensional spaces
\[
S_i=
\begin{cases}
\operatorname{span}\{x_i\},& b_1\le i<b_2,\\
\operatorname{span}\{x_i+y_i\},& b_2\le i\le d_1,\\
\operatorname{span}\{y_i\},& d_1<i\le d_2
\end{cases}
\]
form a submodule $S\cong I[b_1,d_2]$.  The two breaker orientations are
exactly what makes the two transition maps preserve these lines.
\end{remark}

\begin{corollary}[Trichotomy for fully supported embeddings]\label{cor:trichotomy}
Let $U=I[p,q]$ and let $V[p,q]$ be fixed.
\begin{enumerate}
\item Every monomorphism $U\to V[p,q]$ splits if and only if there is no
spanning breaker chain.
\item Both a splitting and a nonsplitting monomorphism $U\to V[p,q]$ exist
if and only if $V[p,q]$ has a direct summand isomorphic to $U$ and also
contains a spanning breaker chain.
\item There exists at least one monomorphism $U\to V[p,q]$, and every such
monomorphism is nonsplitting, if and only if $V[p,q]$ contains a spanning
breaker chain but has no direct summand isomorphic to $U$.
\end{enumerate}
\end{corollary}

\begin{proof}
Part~(1) is Theorem~\ref{thm:main}.  For part~(2), a direct summand
$U\subseteq V[p,q]$ gives a splitting inclusion, while a spanning breaker
chain gives a nonsplitting monomorphism by Theorem~\ref{thm:main}.
Conversely, a splitting monomorphism identifies $U$ with a direct summand,
and a nonsplitting monomorphism gives a spanning breaker chain.

For part~(3), a spanning breaker chain supplies a monomorphism, and the
absence of a direct summand isomorphic to $U$ rules out every splitting
monomorphism.  Conversely, the existence of a nonsplitting monomorphism gives
a spanning breaker chain, while the existence of a direct summand $U$ would
itself give a splitting monomorphism, contrary to the hypothesis.
\end{proof}
We recover Proposition ~ \ref{prop:persistence} as a corollary of the theorem.

\begin{corollary} \label{cor:persistence-splitting}
Let $V[p,q]$ be a restriction of an ordinary persistence module.  Every
monomorphism
\[
I[p,q]\longrightarrow V[p,q]
\]
splits.  Equivalently, every fully supported interval submodule of
$V[p,q]$ is a direct summand.
\end{corollary}

\begin{proof}
All arrows are forward, so no breaker pair, and hence no breaker chain, can
occur.  The conclusion follows from Theorem~\ref{thm:main}.
\end{proof}
\subsection{Construction and Algorithm}
\begin{proposition}[Explicit realization of a breaker chain]\label{prop:explicit-chain}
Let
\[
J_1=I[b_1,d_1],\dots,J_r=I[b_r,d_r]
\]
be a breaker chain spanning $[p,q]$.  Choose, for each $j$, compatible
nonzero generators $x_i^{(j)}\in(J_j)_i$ on its support.  For
$p\le i\le q$, set
\[
z_i=\sum_{\substack{1\le j\le r\\ i\in[b_j,d_j]}}x_i^{(j)}
\qquad\text{and}\qquad
S_i=\operatorname{span}\{z_i\}.
\]
Then $S=(S_i)$ is a submodule of $\bigoplus_{j=1}^rJ_j$ isomorphic to
$I[p,q]$, and its inclusion is nonsplitting.
\end{proposition}

\begin{proof}
By the breaker-chain conditions and Lemma~\ref{lem:hom}(1), for every $j$
there is a nonzero map
\[
\phi_j:I[p,q]\longrightarrow J_j.
\]
Scale the chosen generators so that, at every $i\in[b_j,d_j]$, the map
$\phi_j$ sends the standard generator of $I[p,q]_i$ to $x_i^{(j)}$.  The
direct sum
\[
\Phi=(\phi_1,\dots,\phi_r):I[p,q]\longrightarrow\bigoplus_{j=1}^rJ_j
\]
therefore has image $S_i=\operatorname{span}\{z_i\}$ at vertex $i$.  Since
the chain supports cover $[p,q]$, each $z_i$ is nonzero, so $\Phi$ is a
monomorphism and $S\cong I[p,q]$.  Every $J_j$ is proper, hence the splitting
criterion, Lemma~\ref{lem:splitting}, shows that this inclusion does not
split.
\end{proof}

\begin{proposition}[Barcode Algorithm]\label{prop:algorithm}
Suppose the oriented barcode of $V[p,q]$ contains $m$ interval summands.
Whether $V[p,q]$ contains a nonsplitting copy of $I[p,q]$ can be decided in
$O(m\log m)$ time.  In the same complexity one can produce a spanning breaker
chain of minimum possible length, when one exists.
\end{proposition}

\begin{proof}
By Lemma~\ref{lem:hom}(1), one first filters the $m$ summands to retain only
the proper admissible intervals; this requires only the support endpoints and
the two boundary orientations, so it takes $O(m)$ time once the barcode is
given.  By Proposition~\ref{prop:cover-characterization}(2), the problem is
then exactly whether these intervals cover the discrete interval $[p,q]$.

Sort the retained intervals by increasing left endpoint, breaking ties by
decreasing right endpoint.  Starting with the first uncovered vertex, choose
among all intervals whose left endpoint is at most that vertex one having
maximal right endpoint, and repeat from the next uncovered vertex.  This is
the standard greedy interval-cover algorithm.  It either detects a gap or
returns a cover with the minimum possible number of intervals.  Sorting costs
$O(m\log m)$ and the subsequent sweep is linear.

If a cover is returned, discard any redundant selected interval if necessary;
equivalently, take an inclusion-minimal subcover.  The proof of
Theorem~\ref{thm:main} shows that, when ordered by increasing left endpoint,
such a proper admissible cover is a spanning breaker chain.  For a
minimum-cardinality greedy cover no interval is redundant, so the returned
chain already has minimum length.  Conversely every spanning breaker chain
is a proper admissible cover.  Hence the minimum number of intervals in a
proper admissible cover equals the minimum length of a spanning breaker
chain.
\end{proof}

\section*{Declarations} 

{\bf Ethical Approval}. Not applicable.

{\bf Competing Interest}. Not applicable.

{\bf Funding}. Not applicable.

{\bf Authors' Contributions.} All authors have contributed equally. 

{\bf Availability of data and materials}. Not applicable.

\end{document}